\documentclass[11pt]{article}
\usepackage{amsfonts}
\usepackage{amsmath,amssymb}

\usepackage{pstricks,pstricks-add}

\usepackage{anysize} 
\marginsize{3cm}{3cm}{3cm}{3cm} 

\allowdisplaybreaks
\newtheorem{theorem}{Theorem}

\newtheorem{lemma}[theorem]{Lemma}

\newenvironment{proof}[1][Proof]{\noindent\textbf{#1.} }{\ \rule{0.5em}{0.5em}}

\begin{document}

\title{\bf The set of $p$-harmonic functions in $B_{1}$\\ is total in $C^{k}(\overline{B}_{1})$ }
\author{{\bf Jos\'e Villa-Morales}\\
Universidad Aut\'{o}noma de Aguascalientes\\
Departamento de Matem\'{a}ticas y F\'{\i}sica\\
Av. Universidad No. 940, Cd. Universitaria \\
Aguascalientes, Ags., C.P. 20131, M\'exico\\
\ttfamily{jvilla@correo.uaa.mx}}
\date{}
\maketitle

\begin{abstract}
Let $(-\Delta _{p})^{s}$, with $0<s<1<p<\infty$, be the fractional $p$-Laplacian operator. We prove that the span of $p$-harmonic functions in $B_{1}$ is dense in $C^{k}(\overline{B}_{1})$.

\vspace{0.5cm}
\end{abstract}

\section{Introduction}

Recently much attention has been focused on the study of fractional operators. This is, in part, because these operators are taking an important role in applied mathematics. For example, they arise in fields like molecular biology \cite{SZF}, combustion theory \cite{CRS}, dislocations in mechanical systems \cite{IM}, crystals \cite{Toland} and in models of anomalous growth of certain fractal interfaces \cite{MW}, to name a few. 

Here we are going to study the fractional $p$-Laplacian. Before we start, let us fix some notation. By $B_{r}\subset \mathbb{R}^{d}$ we are going to denote the open ball with center at $0\in \mathbb{R}^{d}$ and radius $r>0$. Let $U\subset \mathbb{R}^{d}$ be an open set, and $k\in \mathbb{N}\cup \{0\}$,
\begin{eqnarray*}
C^{k}(\overline{U})&=&\{\varphi\in C^{k}(U) \ \vert \ D^{\alpha}\varphi \text{ is uniformly continuous on bounded subsets of } U,\\ &&\text{for all } |\alpha| \leq k\},
\end{eqnarray*}
and if $\varphi\in C^{k}(\overline{U})$ is bounded we write $||\varphi||_{C^{k}(\overline{U})}=\sum_{|\alpha|\leq k} ||D^{\alpha}\varphi||_{C(\overline{U})}$, where $C(\overline{U})=C^{0}(\overline{U})$ and  $||\varphi||_{C(\overline{U})}=\sup \{|\varphi(x)|:x\in U\}$.

The fractional $p$-Laplacian operator $(-\Delta _{p})^{s}$, with $0<s<1<p<\infty$, is defined, for $u:\mathbb{R}^{d} \rightarrow \mathbb{R}$ smooth enough, by 
\begin{equation}
(-\Delta _{p})^{s}u(x)=\text{P.V. }\int_{\mathbb{R}^{d}}\frac{|u(x)-u(x+y)|^{p-2}(u(x)-u(x+y))}{|y|^{d+sp}}dy, \label{DefFL}
\end{equation}
where the term P.V. stands for the (Cauchy) principal value. If we want to emphasize the dimension, where the operator is defined, we will write $(-\Delta _{p})_{d}^{s}.$  

We will say that a smooth function $u:\mathbb{R}^{d} \rightarrow \mathbb{R}$ is $p$-harmonic in $B_{1}$ if  $(-\Delta _{p})^{s}u(x)=0$ for each $x\in B_{1}$. 

\begin{theorem}\label{MainTh}
Let $M$ be the set of all $p$-harmonic functions in $B_{1}$. Given $k\in \mathbb{N}\cup \{0\}$, $f\in C^{k}(\overline{B}_{1})$ and $\varepsilon >0$ there exists $u\in\mbox{span} \ (M)$ such that $||f-u||_{C^{k}(\overline{B}_{1})}<\varepsilon .$
\end{theorem}

The fractional $p$-Laplacian represents a natural extension of fractional Laplacian ($p=2$). In the case of the fractional Laplacian the previous result was proved in \cite{DSV}. More precisely, if $p\neq 2 $ the fractional $p$-Laplacian is not a linear operator, but in the case $p=2$ the fractional Laplacian is a linear operator and it is proved in \cite{DSV} that $M$ is dense in $C^{k}(\overline{B}_{1})$.

In the case of fractional Laplacian, there are new proofs of Theorem \ref{MainTh}, see \cite{GSU, Kry, RS}. For more articles concerning the fractional $p$-Laplacian please see, for example,  \cite{CL, DBL, XHW} and the references therein.

The paper is organized as follows. In Section 2 we present some preliminary facts and in Section 3 we give the proof of Theorem \ref{MainTh}, which is based on \cite{BV, DSV, Val}.

\section{Preliminaries}

By $e^{1},...,e^{d}$ let us denote the canonical basis of $\mathbb{R}^{d}$ and by $\langle \cdot, \cdot\cdot\rangle$ the usual inner product in $\mathbb{R}^{d}$. Let us also introduce the function $W_{1}:\mathbb{R}\rightarrow \mathbb{R}$ as
\begin{equation}
W_{1}(t)=(\max \{0,t\})^{s} \label{defW1}
\end{equation}
and the sign function  $\sigma:\mathbb{R}\rightarrow \mathbb{R}$ by
\begin{equation*}
\sigma (t)=\left\{ 
\begin{array}{cc}
1, & t\geq 0, \\ 
-1, & t<0.
\end{array}
\right. 
\end{equation*}
Since $t=|t|\sigma (t)$ we have, for each $r\in \mathbb{R}$,
\begin{eqnarray*}
W_{1}(tr) &=&(\max \{0,|t|\sigma (t)r\})^{s}\\
&=&|t|^{s}W_{1}(\sigma (t)r).
\end{eqnarray*}
Let us also introduce the function  $W_{d}:\mathbb{R}^{d}\rightarrow \mathbb{R}$ as
\begin{equation*}
W_{d}(x)=W_{1}(\langle x,e^{d}\rangle ).
\end{equation*}

In what follows we will give a representation of $(-\Delta _{p})_{d}^{s}W_{d}$ in terms of $(-\Delta _{p})_{1}^{s}W_{1}$, a similar expression, in the case $p=2$, can be seen in \cite{Dyda}.

\begin{lemma}\label{Lmrepd}
If $d>1$, then, for each $x\in \mathbb{R}^{d}$,
\begin{equation*}
(-\Delta _{p})_{d}^{s}W_{d}(x)=\frac{a_{d-1}}{2}B\left( \frac{d-1}{2},\frac{sp+1}{2}\right) |\langle x,e^{d}\rangle |^{-s}(-\Delta_{p})_{1}^{s}W_{1}(\sigma (\langle x,e^{d}\rangle )),
\end{equation*}
where $a_{d-1}$ is the ($d-2$)-dimensional Lebesgue measure of the unit
sphere in $\mathbb{R}^{d-1}$ and $B$ is the usual Beta function.
\end{lemma}

\begin{proof}
If $x=(x_{1},...,x_{d})$, with $x_{d}\neq 0$, and $y=(y_{1},...,y_{d})$ then
\begin{eqnarray*}
W_{d}(x+y) &=&W_{1}(x_{d}(1+(x_{d})^{-1}y_{d})) \\
&=&|x_{d}|^{s}W_{1}(\sigma (x_{d})(1+(x_{d})^{-1}y_{d})) \\
&=&|x_{d}|^{s}W_{1}(\sigma (x_{d})+|x_{d}|^{-1}y_{d}),
\end{eqnarray*}
therefore
\begin{eqnarray*}
(-\Delta _{p})^{s}W_{d}(x)
 &=&\text{P.V.}\int_{\mathbb{R}^{d}}\frac{
|W_{d}(x)-W_{d}(x+y)|^{p-2}(W_{d}(x)-W_{d}(x+y))}{|y|^{d+sp}}dy \\
&=&|x_{d}|^{sp-s}\text{ P.V.}\int_{\mathbb{R}^{d}}\frac{|W_{1}(\sigma
(x_{d}))-W_{1}(\sigma(x_{d})+|x_{d}|^{-1}y_{d})|^{p-2}}{|\sum_{i=1}^{d-1}(y_{i})^{2}+(y_{d})^{2}|^{(d+sp)/2}}\\
&&\times (W_{1}(\sigma(x_{d}))-W_{1}(\sigma (x_{d})+|x_{d}|^{-1}y_{d}))dy\\
&=&|x_{d}|^{sp-s}\text{ P.V.}\int_{\mathbb{R}}\int_{\mathbb{R}^{d-1}}\frac{|W_{1}(\sigma (x_{d}))-W_{1}(\sigma(x_{d})+|x_{d}|^{-1}y_{d})|^{p-2}}{|y_{d}|^{d+sp}\left\vert |\frac{1}{y_{d}}\xi
|^{2}+1\right\vert ^{(d+sp)/2}}\\
&&\times (W_{1}(\sigma (x_{d}))-W_{1}(\sigma
(x_{d})+|x_{d}|^{-1}y_{d}))d\xi dy_{d}.
\end{eqnarray*}
Introducing the change of variable $\zeta =(y_{d})^{-1}\xi $ we get
\begin{eqnarray*}
(-\Delta _{p})_{d}^{s}W_{d}(x)&=&|x_{d}|^{-1-s}\text{ P.V.}\int_{\mathbb{R}}\frac{|W_{1}(\sigma (x_{d}))-W_{1}(\sigma (x_{d})+|x_{d}|^{-1}y_{d})|^{p-2}}{||x_{d}|^{-1}y_{d}|^{1+sp}}\\
&&\times(W_{1}(\sigma (x_{d}))-W_{1}(\sigma(x_{d})+|x_{d}|^{-1}y_{d}))dy_{d} \int_{\mathbb{R}^{d-1}}\frac{d\zeta }{\left\vert |\zeta |^{2}+1\right\vert ^{(d+sp)/2}},
\end{eqnarray*}
and the change of variable $r=|x_{d}|^{-1}y_{d}$ yields
\begin{eqnarray*}
(-\Delta _{p})_{d}^{s}W_{d}(x)&=&|x_{d}|^{-s}\text{ P.V.}\int_{\mathbb{R}}
\frac{|W_{1}(\sigma (x_{d}))-W_{1}(\sigma (x_{d})+r))|^{p-2}}{|r|^{1+sp}}\\
&&\times (W_{1}(\sigma
(x_{d}))-W_{1}(\sigma (x_{d})+r))dr
\int_{\mathbb{R}^{d-1}}
\frac{d\zeta }{\left\vert |\zeta |^{2}+1\right\vert ^{(d+sp)/2}}.
\end{eqnarray*}
Since $d>1$, then
\begin{eqnarray*}
\int_{\mathbb{R}^{d-1}}\frac{d\zeta }{\left\vert |\zeta |^{2}+1\right\vert
^{(d+sp)/2}} &=&a_{d-1}\int_{0}^{\infty }t^{d-2}(t^{2}+1)^{-(d+sp)/2}dt \\
&=&\frac{a_{d-1}}{2}B\left( \frac{d-1}{2},\frac{sp+1}{2}\right) ,
\end{eqnarray*}
in the last equality we have used the change of variable $r=t^{2}/(1+t^{2})$ to get the usual definition of the Beta function (see \cite{Warma}). From this the results follows. \hfill
\end{proof}

\bigskip
Now let us find an elementary limit, essential in the evaluation of $(-\Delta_{p})_{1}^{s}W_{1}(1)$.
\begin{lemma}\label{calimt}
If $0<s\neq 1$ and $p\in \mathbb{R}$, then
\begin{equation*}
\lim_{\varepsilon \downarrow 0}\frac{(1-(1-\varepsilon
)^{s})^{p-1}-((1+\varepsilon )^{s}-1)^{p-1}}{\varepsilon ^{p}}%
=s^{p-1}(p-1)(1-s).
\end{equation*}
\end{lemma}

\begin{proof}
Let us consider $p\neq 1$. Changing of variable $x=1/\varepsilon$ we want to calculate
\begin{eqnarray*}
l &=&\lim_{x\rightarrow \infty }x^{s+p-ps}\left[
(x^{s}-(x-1)^{s})^{p-1}-((x+1)^{s}-x^{s})^{p-1}\right]  \\
&=&\lim_{x\rightarrow \infty }s^{p-1}x\left[ \left( \int_{0}^{1}\left( 1+
\frac{z-1}{x}\right) ^{s-1}dz\right) ^{p-1}-\left( \int_{0}^{1}\left( 1+
\frac{z}{x}\right) ^{s-1}dz\right) ^{p-1}\right]  \\
&=&\lim_{x\rightarrow \infty }s^{p-1}(p-1)x\int_{\int_{0}^{1}\left( 1+\frac{z
}{x}\right) ^{s-1}dz}^{\int_{0}^{1}\left( 1+\frac{z-1}{x}\right)
^{s-1}dz}r^{p-2}dr \\
&=&\lim_{x\rightarrow \infty }s^{p-1}(p-1)x\int_{0}^{\int_{0}^{1}\left( 1+
\frac{z-1}{x}\right) ^{s-1}dz-\int_{0}^{1}\left( 1+\frac{z}{x}\right)
^{s-1}dz}\left( r+\int_{0}^{1}\left( 1+\frac{z}{x}\right) ^{s-1}dz\right)
^{p-2}dr \\
&=&\lim_{x\rightarrow \infty }s^{p-1}(p-1)\int_{0}^{x\left[
\int_{0}^{1}\left( 1+\frac{z-1}{x}\right) ^{s-1}dz-\int_{0}^{1}\left( 1+
\frac{z}{x}\right) ^{s-1}dz\right] }\left( \frac{y}{x}+\int_{0}^{1}\left( 1+
\frac{z}{x}\right) ^{s-1}dz\right) ^{p-2}dy
\end{eqnarray*}
in the last equality we have used the change of variable $y=xr$. Now let us work with
\begin{eqnarray*}
\int_{0}^{1}\left( 1+\frac{z-1}{x}\right)^{s-1}dz-\int_{0}^{1}\left( 1+\frac{z}{x}\right) ^{s-1}dz
&=&(1-s)\int_{0}^{1}\int_{1+\frac{z-1}{x}}^{1+\frac{z}{x}}r^{s-2}drdz \\
&=&\frac{1-s}{x}\int_{0}^{1}\int_{0}^{1}\left( 1+\frac{y+z-1}{x}\right) ^{s-2}dydz,
\end{eqnarray*}
in the last equality we have used the change of variable $y=x\left[ r-\left( 1+\frac{z-1}{x}\right) \right] $. Then 
\begin{equation*}
l =\lim_{x\rightarrow \infty
}s^{p-1}(p-1)\int_{0}^{(1-s)\int_{0}^{1}\int_{0}^{1}\left( 1+x^{-1}(y+z-1)\right) ^{s-2}dydz}\left( \frac{y}{x}+\int_{0}^{1}\left( 1+\frac{z}{x}\right) ^{s-1}dz\right) ^{p-2}dy,
\end{equation*}
and the limit follows from the dominated convergence theorem. \hfill
\end{proof}

\bigskip
To calculate $(-\Delta _{p})_{1}^{s}W_{1}(1)$ we follows the method introduced in \cite{BV}.
\begin{lemma}\label{Lmrepu}
Let $W_{1}$ be defined as in (\ref{defW1}), then $(-\Delta _{p})_{1}^{s}W_{1}(1)=0$.
\end{lemma}

\begin{proof}
From the definition (\ref{DefFL}) we have 
\begin{eqnarray*}
(-\Delta _{p})_{1}^{s}W_{1}(1) 
&=&\int_{-\infty }^{-1}\frac{1}{|r|^{1+sp}}dr-\int_{1}^{\infty }\frac{%
((1+r)^{s}-1)^{p-1}}{|r|^{1+sp}}dr \\
&&+\text{P.V. }\int_{-1}^{1}\frac{|1-(1+r)_{+}^{s}|^{p-2}(1-(1+r)_{+}^{s})}{%
|r|^{1+sp}} \\
&=&I_{1}-I_{2}+I_{3}.
\end{eqnarray*}
Let us calculate each integral. For the first integral $I_{1}=(sp)^{-1}$. In the second integral we use integration by parts to get
\begin{eqnarray*}
I_{2} &=&-\frac{1}{sp}\int_{1}^{\infty }((1+r)^{s}-1)^{p-1}d(r^{-sp}) \\
&=&\frac{1}{sp}\left\{ (2^{s}-1)^{p-1}+s(p-1)\int_{1}^{\infty }\frac{(1+r)^{s-1}((1+r)^{s}-1)^{p-2}}{r^{sp}}dr\right\} .
\end{eqnarray*}

Now, for the third integral we use integration by parts and Lemma \ref{calimt}
\begin{eqnarray*}
I_{3} &=&\lim_{\varepsilon \downarrow 0}\left( \int_{-1}^{-\varepsilon }
\frac{(1-(1+r)^{s})^{p-1}}{|r|^{1+sp}}dr-\int_{\varepsilon }^{1}\frac{
((1+r)^{s}-1)^{p-1}}{r^{1+sp}}dr\right)  \\
&=&\lim_{\varepsilon \downarrow 0}\int_{\varepsilon }^{1}\frac{
(1-(1-r)^{s})^{p-1}-((1+r)^{s}-1)^{p-1}}{r^{1+sp}}dr \\
&=&-\frac{1}{sp}\lim_{\varepsilon \downarrow 0}\int_{\varepsilon }^{1}\left[
(1-(1-r)^{s})^{p-1}-((1+r)^{s}-1)^{p-1}\right] d(r^{-sp}) \\
&=&-\frac{1}{sp}\left\{1-(2^{s}-1)^{p-1}-s(p-1)\int_{0}^{1}\frac{(1-(1-r)^{s})^{p-2}(1-r)^{s-1}}{r^{sp}}dr \right. \\
&&+\left. s(p-1)\int_{0}^{1}\frac{((1+r)^{s}-1)^{p-2}(1+r)^{s-1}}{r^{sp}}dr \right \}.
\end{eqnarray*}
Using the change of variable $t=r/(1-r)$ we get
\begin{equation*}
\int_{0}^{1}\frac{(1-(1-r)^{s})^{p-2}(1-r)^{s-1}}{r^{sp}}dr=\int_{0}^{\infty
}\frac{(1+t)^{s-1}((1+t)^{s}-1)^{p-2}}{t^{sp}}dt.
\end{equation*}
Substituting this in $I_{3}$ we obtain
\begin{equation*}
I_{3}=\frac{1}{sp}\left\{(2^{s}-1)^{p-1}-1+s(p-1)\int_{1}^{\infty }\frac{(1+r)^{s-1}((1+r)^{s}-1)^{p-2}}{r^{sp}}dr\right\}.
\end{equation*}
Adding the three integrals we get the desired result.\hfill
\end{proof}

\section{Proof of the main result}

Let $\xi \in \mathbb{R}^{d}\backslash \{0\}$. If $\xi \neq -|\xi |e^{d}$ we
consider the reflection $R_{\xi }:\mathbb{R}^{d}\rightarrow \mathbb{R}^{d}$,%
\begin{equation*}
R_{\xi }(x)=\frac{2\langle |\xi |e^{d}+\xi ,x\rangle }{|\ |\xi |e^{d}+\xi \
|^{2}}(|\xi |e^{d}+\xi )-x,
\end{equation*}
with respect to the hyperplane $H_{\xi }=\{x\in \mathbb{R}^{d}:\langle
x,|\xi |e^{d}-\xi \rangle =0\}$. On the other hand, if $\xi =-|\xi |e^{d}$
then we consider the reflection $R_{\xi }:\mathbb{R}^{d}\rightarrow \mathbb{R}%
^{d}$,
\begin{equation*}
R_{\xi }(x)=x-2\langle x,e^{d}\rangle e^{d},
\end{equation*}
with respect to the hyperplane $H_{\xi }=\{x\in \mathbb{R}^{d}:\langle
x,e^{d}\rangle =0\}$. In any case, we have 
\begin{equation}
R_{\xi }(\xi )=|\xi |e^{d} \label{pr1}
\end{equation}
 and moreover 
\begin{equation}
\langle R_{\xi }(x),R_{\xi }(y)\rangle =\langle x,y\rangle, \ \ \text{for
all }x,y\in \mathbb{R}^{d}. \label{pr2}
\end{equation}

\begin{proof}[Proof of Theorem \ref{MainTh}]
For each $\xi \in \mathbb{R}^{d}\backslash \{0\}$ let us consider the
function $H_{\xi }:\mathbb{R}^{d}\rightarrow \mathbb{R}$ defined as
\begin{equation*}
H_{\xi }(x)=\left( \max \left\{ 0,\left\langle \frac{\xi }{|\xi |^{2}}+x,\xi
\right\rangle \right\} \right) ^{s}.
\end{equation*}
From (\ref{pr1}) and (\ref{pr2}) we have
\begin{eqnarray*}
H_{\xi }(x) &=&\left( \max \left\{ 0,\left\langle R_{\xi }\left( \frac{\xi }{
|\xi |^{2}}+x\right) ,R_{\xi }(\xi )\right\rangle \right\} \right) ^{s} \\
&=&|\xi |^{s}\left( \max \left\{ 0,\left\langle R_{\xi }\left( \frac{\xi }{
|\xi |^{2}}+x\right) ,e^{d}\right\rangle \right\} \right) ^{s} \\
&=&|\xi |^{s}W_{d}\left( R_{\xi }\left( \frac{\xi }{|\xi |^{2}}+x\right)
\right) .
\end{eqnarray*}
By Lemma \ref{Lmrepd} we get 
\begin{eqnarray*}
(-\Delta _{p})_{d}^{s}H_{\xi }(x) &=&|\xi |^{s(p-1)}(-\Delta
_{p})_{d}^{s}W_{d}\left( R_{\xi }\left( \frac{\xi }{|\xi |^{2}}+x\right)
\right)  \\
&=&c|\xi |^{s(p-1)}\left\vert \left\langle R_{\xi }\left( \frac{\xi }{|\xi
|^{2}}+x\right) ,e^{d}\right\rangle \right\vert ^{-s} \\
&&\times (-\Delta _{p})_{1}^{s}W_{1}\left( \sigma \left( \left\langle R_{\xi
}\left( \frac{\xi }{|\xi |^{2}}+x\right) ,e^{d}\right\rangle \right) \right),
\end{eqnarray*} 
where $c>0$ is a constant. By Lemma \ref{Lmrepu}, $(-\Delta _{p})_{d}^{s}H_{\xi }(x)=0$ if $\left\langle R_{\xi }\left( |\xi |^{-2}\xi +x\right) ,e^{d}\right\rangle >0$. Since
\begin{equation*}
\left\langle R_{\xi }\left( \frac{\xi }{|\xi |^{2}}+x\right),e^{d}\right\rangle =
\frac{1}{|\xi|}\left\langle R_{\xi }\left( \frac{\xi }{
|\xi |^{2}}+x\right) ,R_{\xi }(\xi )\right\rangle
=\frac{1}{|\xi|}(1+\langle \xi ,x\rangle),
\end{equation*}
then $(-\Delta _{p})_{d}^{s}H_{\xi }=0$ on $V_{\xi }=\{x\in \mathbb{R}^{d}:\langle \xi ,x\rangle >-1\}$.  In this way, given $\xi \in \mathbb{R}^{d}\backslash \{0\}$, we have $H_{\xi }(x)=\left( \left\langle |\xi |^{-2}\xi +x,\xi \right\rangle\right) ^{s}=\left(1+\left\langle x,\xi\right\rangle \right) ^{s}$ and 
\begin{equation}
D^{\alpha }H_{\xi }(x)=s(s-1)\cdots (s-|\alpha |+1)\left( 1+\left\langle x,\xi\right\rangle \right) ^{s-|\alpha |}\xi ^{\alpha },
\label{derideH}
\end{equation}
for each $x\in V_{\xi }$ and $\alpha \in (\mathbb{N}\cup \{0\})^{d}$. 

Now let us consider the linear space
\begin{equation*}
\mathcal{V}=\{v:\mathbb{R}^{d}\rightarrow \mathbb{R},\ v \text{ is smooth and }p\text{-harmonic in some neighborhood of }0\}.
\end{equation*}
For each $m\in \mathbb{N}\cup \{0\}$ we are going to denote by $N_{m}$ the number of elements of the set $I_{m}=\{\alpha \in (\mathbb{N}\cup \{0\})^{d}:|\alpha|\leq m\}$. Let us enumerate $I_{m}$ as $\{\beta _{1},...,\beta _{N_{m}}\}$ and define the subset $\mathcal{V}_{N_{m}}$ of $\mathbb{R}^{N_{m}}$ as
\begin{equation*}
\mathcal{V}_{N_{m}}=\{(D^{\beta _{1}}v(0),...,D^{\beta _{N_{m}}}v(0)):v\in \mathcal{V}\}.
\end{equation*}
The set $\mathcal{V}_{N_{m}}$ is a linear subspace of $\mathbb{R}^{N_{m}}$. We claim that $\mathcal{V}_{N_{m}}=\mathbb{R}^{N_{m}}$. To prove this we
proceed by contradiction, as in \cite{Val}. Suppose that $\mathcal{V}_{N_{m}}\subsetneqq \mathbb{R}^{N_{m}}$. Thus there exists $u=(u_{1},...,u_{N_{m}})\in \mathbb{R}^{N_{m}}$, with $|u|=1$, such that $\mathcal{V}_{N_{m}}$ is contained in the hyperplane $\{x\in \mathbb{R}^{N_{m}}:\langle x,u\rangle =0\}$. 

By the previous discussion we know, that for each $\xi \in \mathbb{R}^{d}\backslash \{0\}$,
\begin{equation*}
(D^{\beta _{1}}H_{\xi }(0),...,D^{\beta _{N_{m}}}H_{\xi }(0))\in \mathcal{V}_{N_{m}},
\end{equation*}
therefore (\ref{derideH}) implies
\begin{equation*}
0=\sum_{i=1}^{N_{m}}u_{i}s(s-1)\cdots (s-|\beta _{i}|+1)\xi ^{\beta _{i}},\ \text{for} \ \xi \in \mathbb{R}^{d}\backslash \{0\}.
\end{equation*}
In this way (because $\mathbb{R}^{d}\backslash \{0\}$ is an open set, see \cite{DSV})
\begin{equation*}
u_{i}s(s-1)\cdots (s-|\beta _{i}|+1)=0,\text{ for }i\in \{1,...,N_{m}\},
\end{equation*}
but $s\in (0,1)$ implies $u=0$, contradicting $|u|=1$. 

Now let us see the set of $p$-harmonic functions is total in $C^{k}(\overline{B}_{1})$. Let $f\in C^{k}(\overline{B}_{1})$ and $\varepsilon >0$. By a density theorem (see, for example, Corollary 6.3 and Proposition 7.1 in the Appendixes of \cite{EK}) there exists a polynomial 
\begin{equation*}
p_{\varepsilon }(x)=\sum_{i=1}^{n_{\varepsilon }}c_{i}x^{\gamma _{i}},
\end{equation*}
with $c_{i}\in \mathbb{R}$, $\gamma _{i}\in (\mathbb{N}\cup \{0\})^{d}$ such that 
\begin{equation}
||f-p_{\varepsilon }||_{C^{k}(\overline{B}_{1})}<\frac{\varepsilon }{2}.  \label{apxpol}
\end{equation}
Let $i\in \{1,...,n_{\varepsilon }\}$ and take $m_{i}=k+|\gamma _{i}|$. Let
us enumerate $I_{m_{i}}$ as $\{\beta _{1},...,\beta _{N_{m_{i}}}\}$, where $\beta _{N_{m_{i}}}=\gamma _{i}$. Since $(0,...,0,\gamma _{i}!)\in \mathbb{R}^{N_{m_{i}}}$ there exists $v_{i}\in \mathcal{V}$ such that $D^{\alpha }v_{i}(0)=\gamma _{i}!1_{\{\gamma _{i}\}}(\alpha )$, for all $\alpha \in I_{m_{i}}$. Let us suppose, the smooth function, $v_{i}$ is $p$-harmonic in $B_{r_{i}}$. Then, let us consider the function $\tilde{v}_{i}:B_{1}\rightarrow \mathbb{R}$ defined as
\begin{equation*}
\tilde{v}_{i}(x)=\frac{1}{(\tilde{r}_{i})^{|\gamma _{i}|}}v_{i}(\tilde{r}_{i}x),
\end{equation*}
where
\begin{equation*}
\tilde{r}_{i}=\min \left\{ 1,r_{i},\frac{\varepsilon }{2}\left(\sum_{i=1}^{n_{\varepsilon} }|c_{i}|\ c(\gamma _{i})\right) ^{-1}\right\},
\end{equation*}
and $c(\gamma _{i})$ is defined in (\ref{defci}). The function $\tilde{v}_{i}$ is $p$-harmonic in $B_{1}$ and $D^{\alpha }\tilde{v}_{i}(0)=\gamma _{i}!1_{\{\gamma _{i}\}}(\alpha ),\ $for all $\alpha
\in I_{m_{i}}$. Let $g_{i}(x)=\tilde{v}_{i}(x)-x^{\gamma _{i}}$, $x\in B_{1}$, thus $D^{\alpha }g_{i}(0)=0$, for all $\alpha \in I_{m_{i}}$. 

If $\alpha \in (\mathbb{N}\cup \{0\})^{d}$ with $|\alpha |\leq k$, the Taylor theorem applied to $D^{\alpha }g_{i}$ at $0$ yields
\begin{equation*}
D^{\alpha }g_{i}(x)=\sum_{|\gamma |\leq |\gamma _{i}|}\frac{D^{\alpha
+\gamma }g_{i}(0)}{\gamma !}x^{\gamma }+\sum_{|\gamma |=|\gamma _{i}|+1}%
\frac{D^{\alpha +\gamma }g_{i}(tx)}{\gamma !}x^{\gamma },
\end{equation*}%
for some $t\in (0,1)$. If $|\gamma |=|\gamma _{i}|+1$, then $|\alpha +\gamma
|>|\gamma _{i}|$, thus if $z\in B_{1}$%
\begin{eqnarray*}
D^{\alpha +\gamma }g_{i}(z) &=&D^{\alpha +\gamma }\tilde{v}_{i}(z) \\
&=&(\tilde{r}_{i})^{|\alpha +\gamma |-|\gamma _{i}|}(D^{\alpha +\gamma
})v_{i}(\tilde{r}_{i}z) \\
&=&(\tilde{r}_{i})^{|\alpha |+1}(D^{\alpha +\gamma })v_{i}(\tilde{r}_{i}z).
\end{eqnarray*}
In particular, for $z=tx$, $x\in B_{1}$, we obtain  (using $|\alpha+\gamma|\leq m_{i}$, when $|\gamma|\leq |\gamma_{i}|$)
\begin{eqnarray}
|D^{\alpha }g_{i}(x)| &=&\left\vert \sum_{|\gamma |=|\gamma _{i}|+1}\frac{
D^{\alpha +\gamma }g_{i}(tx)}{\gamma !}x^{\gamma }\right\vert  \nonumber \\
&\leq &\sum_{|\gamma |=|\gamma _{i}|+1}\frac{(\tilde{r}_{i})^{|\alpha
|+1}||v_{i}||_{C^{|\alpha |+|\gamma _{i}|+1}(\overline{B}_{r_{i}})}}{\gamma !} \nonumber \\
&\leq &\left( \sum_{|\gamma |=|\gamma _{i}|+1}\frac{|}{\gamma !}\right)
||v_{i}||_{C^{k+|\gamma _{i}|+1}(\overline{B}_{r_{i}})}\tilde{r}_{i}=:c(\gamma _{i})\tilde{r}_{i}, \label{defci}
\end{eqnarray}
we have used $\tilde{r}_{i}\in (0,1)$. Then, by (\ref{apxpol}), 
\begin{eqnarray*}
\left\Vert f-\sum_{i=1}^{n_{\varepsilon }}c_{i}\tilde{v}_{i}\right\Vert
_{C^{k}(\overline{B}_{1})} &\leq &\frac{2}{\varepsilon }+\left\Vert
\sum_{i=1}^{n_{\varepsilon }}c_{i}x^{\gamma _{i}}-\sum_{i=1}^{n_{\varepsilon
}}c_{i}\tilde{v}_{i}\right\Vert _{C^{k}(\overline{B}_{1})} \\
&\leq &\frac{2}{\varepsilon }+\sum_{i=1}^{n_{\varepsilon }}|c_{i}|\left\Vert
g_{i}\right\Vert _{C^{k}(\overline{B}_{1})} \\
&<&\varepsilon .
\end{eqnarray*}
Therefore $f$ is approximated by $\sum_{i=1}^{n_{\varepsilon }}c_{i}\tilde{v}_{i}$, that belongs to the span of $p$-harmonic functions in $B_{1}$.\hfill
\end{proof}

\section*{Acknowledgment} The author was partially supported by the grant PIM20-1 of Universidad Aut\'{o}noma de Aguascalientes.

\end{document}